\documentclass[12pt,bezier]{article}
\usepackage{amssymb}
\usepackage{mathrsfs}
\usepackage{amsmath}
\usepackage{amsfonts,amsthm,amssymb}
\usepackage{amsfonts}
\usepackage{graphics}
\usepackage{graphicx}
\usepackage{color}
\usepackage{ifpdf}
\usepackage{fancyhdr}
\usepackage{epstopdf}
\usepackage{epsfig}
\usepackage{indentfirst}
\usepackage{CJK}
\usepackage{float}

\newtheorem{lem}{Lemma}[section]
\newtheorem{thm}[lem]{Theorem}

\begin{document}

\title{A note on $Oct_{1}^{+}$-free graphs and $Oct_{2}^{+}$-free graphs}

\author{ Wenyan Jia$^1$, \quad Shuai Kou$^1$, \quad Chengfu Qin$^2$, \quad Weihua Yang$^1$\footnote{Corresponding author. E-mail: ywh222@163.com,~yangweihua@tyut.edu.cn}\\
\\ \small $^1$Department of Mathematics, Taiyuan University of
Technology,\\
\small  Taiyuan Shanxi-030024,
China\\
\small $^2$Department of Mathematics, Guangxi Teachers Education University, Nanning}
\date{}
\maketitle

{\small{\bf Abstract:}
Let $Oct_{1}^{+}$ and $Oct_{2}^{+}$ be the planar and non-planar graphs that obtained from the Octahedron by 3-splitting a vertex respectively. For $Oct_{1}^{+}$, we prove that a 4-connected graph is $Oct_{1}^{+}$-free if and only if it is $C_{6}^{2}$, $C_{2k+1}^{2}$ $(k \geq 2)$ or it is obtained from $C_{5}^{2}$ by repeatedly 4-splitting vertices. We also show that a planar graph is $Oct_{1}^{+}$-free if and only if it is constructed by repeatedly taking 0-, 1-, 2-sums starting from $\{K_{1}, K_{2} ,K_{3}\} \cup \mathscr{K} \cup \{Oct,L_{5} \}$, where $\mathscr{K}$ is the set of graphs obtained by repeatedly taking the special 3-sums of $K_{4}$. For $Oct_{2}^{+}$, we prove that a 4-connected graph is $Oct_{2}^{+}$-free if and only if it is planar, $C_{2k+1}^{2}$ $(k \geq 2)$, $L(K_{3,3})$ or it is obtained from $C_{5}^{2}$ by repeatedly 4-splitting vertices.


\section{Introduction}

Let $Oct$ denote the Octahedron. We only consider simple graphs in this article. Let $G$ and $H$ be two graphs. $H$ is called a minor of $G$ denoted by $H \leq_{m} G$ if it can be generated by deleting or contracting edges from $G$. If $G$ has no minor isomorphic to $H$, $G$ is $H$-free. Assume $v$ is a vertex of a 3-connected graph $G$ such that $d(v) \geq 4$. Given two sets $A,B\subseteq N_{G}(v)$, where $N_{G}(v)$ is the set of vertices adjacent to $v$ in $G$ and $A \cap B=\emptyset$, $min\{|A|,|B|\}\geq 2$. We mean a $3$-$split$ of $v$ is the operation of first deleting $v$ from $G$ and adding two new adjacent vertices $v'$, $v''$, then joining $v'$ to vertices in $A$ and $v''$ to vertices in $B$. It is clearly that a graph obtained by 3-splitting a vertex of a 3-connected graph will also be 3-connected.

For a given graph $H$, characterizing $H$-free graphs is a difficult topic in graph theory. We focus on 3-connected graph $H$ in this paper. Tutte's Wheel Theorem states that every 3-connected graph can be obtained from a wheel by repeatedly adding edges and 3-splitting vertices~\cite{W.T.Tutte}. By this theorem, we can generate all 3-connected graphs.

Let $G_{1}$, $G_{2}$ be two disjoint graphs with more than $k$ vertices. The 0-sum of $G_{1}$ and $G_{2}$ is the disjoint union of $G_{1}$, $G_{2}$. The 1-sum of $G_{1}$, $G_{2}$ is obtained by identifying one vertex of $G_{1}$ with one vertex of $G_{2}$. The 2-sum of $G_{1}$, $G_{2}$ is obtained by identifying one edge of $G_{1}$ with one edge of $G_{2}$, and the common edge could be deleted after identification. The 3-sum of $G_{1}$, $G_{2}$ is obtained by identifying one triangle of $G_{1}$ with one triangle of $G_{2}$, and some of the three common edges could be deleted after identification.

Next, we introduce some known results for $H$-free graphs where $H$ is 3-connected and we order the results according to the number of edges of $H$.

Ding~\cite{G.Ding} characterized all $H$-free graphs for 3-connected $H$ with at most 11 edges, including $Oct\backslash e$. Let $\aleph$ denote the set of graphs obtained by 3-summing wheels and Prisms, and let $K_{5}^{\triangle}$ denote the graph obtained by 3-summing Prism and $K_{5}$.

\begin{thm}[\cite{G.Ding}]\label{thm1.1}
$Oct\backslash e$-free graphs consists of graphs in $\aleph$ and 3-connected minors of $V_{8}$, $Cube$, and $K_{5}^{\triangle}$.
\end{thm}

For 3-connected graphs with 12 edges, $V_{8}$-free graphs, $Cube$-free graphs and $Oct$-free graphs are characterized in~\cite{G.Ding Oct,J.Maharry and N.Robertson,J.Maharry Cube}.

\begin{thm}[\cite{G.Ding Oct}]\label{thm1.2}
A graph is Oct-free if and only if it is constructed by 0-, 1-, 2- and 3-sums starting from graphs in $\{K_{1}, K_{2}, K_{3}, K_{4}\} \cup \{C_{2n-1}^{2} : n\geq 3\} \cup \{L_{4}^{'}, L_{5}, L_{5}^{'}, L_{5}^{''}, P_{10}\}$ (see Figure 1).
\end{thm}

\input{Figure1.Tpx}

There are 51 3-connected graphs with 13 edges, but only two related results. One is for 4-connected $Oct^{+}$-free graphs, where $Oct^{+}$ denote the graph $Oct+e$~\cite{J.Maharry}. It can be seen that $Oct^{+}$ is isomorphic to $K_{6}$ with two parallel edges removed.

\begin{thm}[\cite{J.Maharry}]\label{thm1.3}
Every 4-connected graph that does not contain a minor isomorphic to $Oct^{+}$ is either planar or the square of an odd cycle.
\end{thm}

In this paper, we consider the two graphs that obtained by 3-splitting a vertex of the Octahedron. We denote the planar one by $Oct_{1}^{+}$ and the non-planar one by $Oct_{2}^{+}$ (as shown in Figure 2). It can be seen that $Oct_{1}^{+}$ and $Oct_{2}^{+}$ are 3-connected and they both have 13 edges. Our purpose is to characterize 4-connected $Oct_{1}^{+}$-free graphs and 4-connected $Oct_{2}^{+}$-free graphs. For $Oct_{1}^{+}$, we also characterize all planar $Oct_{1}^{+}$-free graphs by characterizing 3-connected planar $Oct_{1}^{+}$-free graphs.

\begin{figure}[htbp]
\centering
\includegraphics[height=3cm,width=6cm,angle=0]{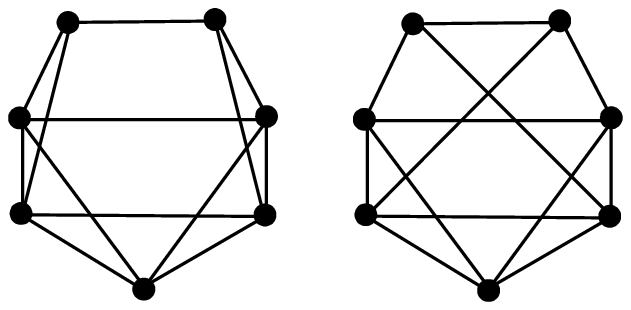}
\caption{$Oct_{1}^{+}$ , $Oct_{2}^{+}$}
\label{fig2}
\end{figure}

Let $v$ be a vertex of a 4-connected graph $G$. A 4-$split$ of $v$ produces a new graph $G'$ as follows. Given two sets $A,B\subseteq N_{G}(v)$, where $N_{G}(v)$ is the set of vertices adjacent to $v$ in $G$ and $min\{|A|,|B|\}\geq 3$. Remove $v$ from $G$ and add two new adjacent vertices $v'$, $v''$ such that $N_{G'}(v')=A\cup \{v''\}$, $N_{G'}(v'')=B\cup \{v'\}$. Clearly, $G'$ is also 4-connected.

 The following are the main results of this paper.

\begin{thm}\label{thm1.4}
A 4-connected graph is $Oct_{1}^{+}$-free if and only if it is $C_{6}^{2}$, $C_{2k+1}^{2}$ ($k\geq2$) or it is obtained from $C_{5}^{2}$ by repeatedly 4-splitting vertices. And $C_{6}^{2}$ is the only 4-connected planar $Oct_{1}^{+}$-free graph.
\end{thm}

\begin{thm}\label{thm1.5}
A planar graph is $Oct_{1}^{+}$-free if and only if it is constructed by repeatedly taking 0-, 1-, 2-sums starting from $\{K_{1}, K_{2} ,K_{3}\} \cup \mathscr{K} \cup \{Oct,L_{5} \}$, where $\mathscr{K}$ is the set of graphs obtained by repeatedly taking the special 3-sums of $K_{4}$ .
\end{thm}

\input{Figure3.Tpx}

\begin{thm}\label{thm1.6}
A 4-connected graph is $Oct_{2}^{+}$-free if and only if it is planar, $C_{2k+1}^{2}$ ($k\geq2$), $L(K_{3,3})$ or it is obtained from $C_{5}^{2}$ by repeatedly 4-splitting vertices.
\end{thm}

\section{Preliminaries}

In this section, we introduce some definitions and known results we will use in section 3.

A $separation$ of $G$ is an ordered pair of subgraphs $(H,K)$ such that $E(H) \cap E(K) =\emptyset$, $H \cup K =G$, and $|E(H)|\geq |V(H) \cap V(K)|\leq |E(K)|$. A separation $(H,K)$ is called a $k$-$separation$ if $|V(H) \cap V(K)|=k$. A $cyclic$ $separation$ is a separation $(H,K)$ in which both $H$ and $K$ contain circuits. Suppose $k$ is an integer greater than two. A graph $G$ is $cyclically$ $k$-$connected$ if it is 2-connected, $|E(G)|-|V(G)|+1\geq k$ and there does not exist a cyclic $k'$-separation of $G$ for $k' \leq k$.

A graph is $cubic$ if it is 3-regular. The graph $L(H)$ is called the $line$ $graph$ of $G$ if $V(L(G)) = E(G)$, and for any two vertices $e$, $f$ in $V(L(G))$, $e$ and $f$ are adjacent vertices if and only if they are adjacent edges in $G$. Let $\mathcal{C}=\{C_{n}^{2}:n\geq5\}$ and $\mathcal{L}=\{L(H)$: $H$ be the cubic cyclically 4-connected graph$\}$. A $(G_{0}, G_{n})$-$chain$ is a sequence of 4-connected graphs $G_{0}, G_{1},...,G_{n}$ and each $G_{i}$ $(i < n)$ has an edge $e_{i}$ such that $G_{i}/e_{i}=G_{i+1}$.

There is a classical result of Martinov for 4-connected graphs, which is known as chain theorem.

\begin{thm}[\cite{N.Martinov}]\label{thm2.1}
For every 4-connected graph $G$, there exists a sequence of 4-connected graphs $G_{0}, G_{1},...,G_{n}$ such that $G_{0}=G$, $G_{n} \in \mathcal{C} \cup \mathcal{L}$, and every $G_{i}$ $(i < n)$ has an edge $e_{i}$ for which $G_{i}/e_{i}=G_{i+1}$.
\end{thm}

This result has been strengthened by Qin and Ding as follows.

\begin{thm}[\cite{C.Qin}]\label{thm2.2}
Let $G$ be a 4-connected graph not in $\mathcal{C} \cup \mathcal{L}$. If $G$ is planar, then there exists a $(G, C_{6}^{2})$-chain; if $G$ is non-planar, then there exists a $(G, K_{5})$-chain.
\end{thm}

Thus, any 4-connected graph that not in $\mathcal{C} \cup \mathcal{L}$ can be generated by repeatedly 4-splitting vertices from $C_{6}^{2}$ or $C_{5}^{2}$.

There are some good properties for cubic cyclically 4-connected graphs and the line graphs. Adding a $handle$ to $G$ means the operation that first subdivide two nonadjacent edges $e_{1}$ and $e_{2}$ of $G$, then add a new edge connecting two new internal vertices. And a graph $H$ is $topologically$ contained in a graph $G$, if there exists a subgraph of $G$ that is isomorphic to a subdivision of $H$.

\begin{lem}[\cite{N.C.Wormald}]\label{lem2.3}
The class of all cubic cyclically 4-connected graphs can be generated by repeatedly adding handles starting from $K_{3,3}$ or the $Cube$.
\end{lem}

\begin{lem}[\cite{G.Ding P}]\label{lem2.4}
If $G$ is a cyclically 4-connected non-planar cubic graph, then either $G=K_{3,3}$ or $G$ contains a subdivision of $V_{8}$.
\end{lem}

\begin{lem}[\cite{J.Maharry four}]\label{lem2.5}
If $H$ is topologically contained in $G$, then $L(H) \leq_{m} L(G)$.
\end{lem}

The following are some results for 3-connected graphs.

\begin{lem}[\cite{G.Ding}]\label{lem2.6}
Let $H$ be a 3-connected graph. Then a graph is $H$-free if and only if it is constructed by repeatedly taking 0-, 1-, and 2-sums, starting from $\{K_{1}, K_{2}, K_{3}\}$ $\cup$ $\{$3-connected $H$-free graphs$\}$.
\end{lem}

\begin{lem}[\cite{P.Seymour}]\label{lem2.7}
Suppose a 3-connected graph $H \neq W_{3}$ is a proper minor of a 3-connected graph $G \neq W_{n}$. Then G has a minor J, which is obtained from H by either adding an edge or 3-splitting a vertex.
\end{lem}

\section{Proof of main results}

\subsection{4-connected $Oct_{1}^{+}$-free graphs}

In this section, we characterize 4-connected $Oct_{1}^{+}$-free graphs and prove Theorem~\ref{thm1.4}.

\begin{lem}\label{lem3.1}
If a 4-connected graph $G \in \mathcal{C} \cup \mathcal{L}$ and $G$ is $Oct_{1}^{+}$-free, then $G$ is $C_{6}^{2}$ or $C_{2k+1}^{2}$ $(k\geq 2)$.
\end{lem}

\begin{proof}
Suppose $G=L(H)$, where $H$ is a cubic cyclically 4-connected graph. By Lemma~\ref{lem2.3}, $H$ can be generated by repeatedly adding handles starting from $K_{3,3}$ or the $Cube$. Thus, $L(H) \geq_{m} L(K_{3,3})$ or $L(H) \geq_{m} L(Cube)$ by lemma~\ref{lem2.5}. Since $L(K_{3,3})$ and $L(Cube)$ both contain $Oct_{1}^{+}$ as a minor (as shown in Figure 4 and Figure 5), $G$ contains $Oct_{1}^{+}$-minor too.

\input{Figure4.Tpx}

\input{Figure5.Tpx}

Thus we assume that $G \in \mathcal{C}$. Suppose $G=C_{2k+1}^{2}$ $(k \geq 2)$, then $G$ is $Oct_{1}^{+}$-free since $C_{2k+1}^{2}$ $(k\geq 2)$ is $Oct$-free. For $C_{2k}^{2}$ $(k \geq 3)$, $C_{2k+2}^{2}$ contains $C_{2k}^{2}$ as a minor. Clearly, $C_{6}^{2}$ is $Oct_{1}^{+}$-free since it only has six vertices. And it is easy to verify that $C_{8}^{2}$ contains $Oct_{1}^{+}$-minor, thus all $C_{2k}^{2}$ contains $Oct_{1}^{+}$-minor for $k \geq 4$.
\end{proof}

\begin{thm}\label{thm3.2}
A 4-connected planar graph is $Oct_{1}^{+}$-free if and only if it is $C_{6}^{2}$.
\end{thm}

\begin{proof}
The sufficiency clearly holds. To prove the necessity, assume that $G$ is a 4-connected planar $Oct_{1}^{+}$-free graph. If $G \in \mathcal{C} \cup \mathcal{L}$, by Lemma~\ref{lem3.1} $G=C_{6}^{2}$. If $G$ is not in $\mathcal{C} \cup \mathcal{L}$, by Theorem~\ref{thm2.2} there exists a $(G, C_{6}^{2})$-chain.

Let $\{v_{1}, v_{2},..., v_{6}\}$ be vertices of $C_{6}^{2}$ (shown in Figure 6). By symmetry, we 4-split $v_{1}$ and first consider the minimal case $|A|=|B|=3$, where $A$, $B$ belongs to $N(v_{1})$, $A \cup B =N(v_{1})$. Suppose $v_{1}'$, $v_{1}''$ are two new vertices obtained by 4-splitting $v_{1}$, and $N(v_{1}')=A \cup \{v_{1}''\}$, $N(v_{1}'')=B \cup \{v_{1}'\}$. Since the four neighbors of $v_{1}$ in $C_{6}^{2}$ construct a 4-cycle, three vertices in $A$ are pairwise adjacent. By symmetry, we assume that $A = \{v_{2}, v_{5}, v_{6}\}$, then $v_{3}$ must be adjacent to $v_{1}''$. Therefore, $B$ must be one of the following sets: $\{v_{3}, v_{5}, v_{6}\}$, $\{v_{2}, v_{3}, v_{5}\}$, $\{v_{2}, v_{3}, v_{6}\}$. In all cases, the new graph $G'$ generated by 4-splitting $v_{1}$ from $C_{6}^{2}$ with $|A|=|B|=3$ contains $Oct_{1}^{+}$. Clearly, all other 4-splits of $C_{6}^{2}$ contain $G'$ as a minor, thus contain $Oct_{1}^{+}$. Hence, $G$ is $C_{6}^{2}$.
\end{proof}

\input{Figure6.Tpx}

\noindent\emph{Proof of Theorem~\ref{thm1.4}}

The result follows from Theorem~\ref{thm2.2}, Lemma~\ref{lem3.1} and Theorem~\ref{thm3.2}.
$\hfill\qedsymbol$

\subsection{Planar $Oct_{1}^{+}$-free graphs}

In this section, we characterize all planar $Oct_{1}^{+}$-free graphs and prove Theorem~\ref{thm1.5}. We first establish some lemmas.

\begin{lem}\label{lem3.3}
If $G_{1}$, $G_{2}$ are $k$-connected for $k=0,1,2,3$ and at least one of them is non-planar, then the $k$-sum of $G_{1}$ and $G_{2}$ is non-planar.
\end{lem}

\begin{proof}
It clearly holds for $k=0,1$. Without loss of generality, we suppose $G_{1}$ is non-planar, $G_{2}$ is planar. Then $G_{1}$ contains a subdivision of $K_{3,3}$ or $K_{5}$, we denote it by $\Gamma$. When $k=2$, let $G$ be the 2-sum of $G_{1}$ and $G_{2}$, let $e=uv$ be the common edge. Suppose $G$ is planar, then $e$ is contained in the subdivision $\Gamma$ and $e$ is deleted after identification. Since $G_{2}$ is 2-connected, there exists an ($u$,$v$)-path $P$ different from $e$. Then $G_{1} \cup P$ contains a subdivision of $K_{3,3}$ or $K_{5}$, a contradiction.

When $k=3$, let $G$ be the graph obtained by 3-summing $G_{1}$ and $G_{2}$ over a common triangle $v_{1}v_{2}v_{3}v_{1}$. Suppose $G$ is planar and $v_{1}v_{2},v_{2}v_{3},v_{1}v_{3}$ are all deleted after identification. Since $G_{1}$ is non-planar, some edges in $\{v_{1}v_{2},v_{2}v_{3},v_{1}v_{3}\}$ are contained in the subdivision $\Gamma$. If $\Gamma$ is a subdivision of $K_{3,3}$, since there is no triangle in $\Gamma$, at most two edges of $\Gamma$ are contained in $\{v_{1}v_{2},v_{2}v_{3},v_{1}v_{3}\}$, say $v_{1}v_{2},v_{2}v_{3}$. Since $G_{2}$ is 3-connected, there exists a vertex $v$ different from $v_{1},v_{2},v_{3}$ and three internally-disjoint $(v,v_{1})$-path, $(v,v_{2})$-path, $(v,v_{3})$-path in $G_{2}$. By contracting $(v,v_{2})$-path to $v_{2}$, we can obtain a $(v_{1},v_{2})$-path $P_{1}$ and a $(v_{2},v_{3})$-path $P_{2}$. Then $\Gamma \setminus v_{1}v_{2} \setminus v_{2}v_{3} \cup P_{1} \cup P_{2}$ forms a subdivision of $K_{3,3}$ again. This contradicts to the planarity of $G$. Thus we assume that $\Gamma$ is a subdivision of $K_{5}$. As shown in Figure 7, the 3-sum $G'$ of $K_{5}$ and $K_{4}$ is non-planar. Since every 3-connected graph contains $W_{3} = K_{4}$ as a minor, $G$ must contain $G'$ as a minor. A contradiction.
\end{proof}

\input{Figure7.Tpx}

It is sufficient to consider the $k$-sums $(k=0,1,2,3)$ of planar graphs to characterize all planar $Oct_{1}^{+}$-free graphs.

\begin{lem}\label{lem3.4}
If $G_{1}$, $G_{2}$ are both planar, then the $k$-sum $G$ $(k=0,1,2)$ of $G_{1}$ and $G_{2}$ is planar.
\end{lem}

\begin{proof}
It clearly holds for $k=0,1$. When $k=2$, let $e$ be the common edge. Since $G_{i}$ $(i=1,2)$ is planar, there exists a planar embedding $H_{i}$ of $G_{i}$ such that the outer face $\widetilde{f_{i}}$ of $H_{i}$ is incident with $e$. Thus a planar embedding of $G$ can be obtained by embedding $H_{2}$ in $\widetilde{f_{1}}$ of $H_{1}$.
\end{proof}

Recall that the 3-sum of $G_{1}$, $G_{2}$ is obtained by identifying one triangle of $G_{1}$ with one triangle of $G_{2}$, and some of the three common edges could be deleted after identification. Next, we define the $special$ $3$-$sum$ of two graphs. A triangle $abca$ of $G$ is called a $separating$ $triangle$ if the graph obtained from $G$ by deleting vertices $a$, $b$, $c$ is disconnected. Otherwise, we call the triangle $abca$ $non$-$separating$ $triangle$. The $special$ $3$-$sum$ of $G_{1}$ and $G_{2}$ is obtained by taking 3-sum of them over a non-separating triangle of both $G_{1}$ and $G_{2}$.

\begin{lem}\label{lem3.5}
Let $G_{1}$, $G_{2}$ be two 3-connected planar graphs with triangles and let $G$ be a 3-sum of them. If $G$ is obtained by taking special 3-sum of them, then $G$ is planar. Otherwise, $G$ is non-planar.
\end{lem}

\begin{proof}
Suppose $G$ is obtained by 3-summing $G_{1}$, $G_{2}$ over an non-separating triangle $C_{1}$ of both $G_{1}$ and $G_{2}$. Let $f$ be the face of $G_{1}$ that bounded by $C_{1}$. Then there exists a planar embedding $H$ of $G_{1}$ such that the outer face $\tilde{f}$ of $H$ has the same boundary as $f$. Thus a planar embedding of $G$ can be obtained by embedding $G_{2}$ in $\tilde{f}$.

Next, we suppose that $G$ is obtained by 3-summing $G_{1}$ and $G_{2}$ over a separating triangle $C_{2}=abca$ of $G_{1}$ or $G_{2}$, say $G_{1}$. Since $G_{1}$ is planar, there exist two vertices $u_{1}$, $u_{2}$ such that $u_{1}$ is in int$C_{2}$ and $u_{2}$ is in ext$C_{2}$. Let $u_{3}$ be a vertex of $G_{2}$ that different from $\{a,b,c\}$. Since both $G_{1}$ and $G_{2}$ are 3-connected, there exists a $\{u_{i},a\}$-path $P_{i1}$, a $\{u_{i},b\}$-path $P_{i2}$ and a $\{u_{i},c\}$-path $P_{i3}$ in $G$ such that $P_{i1}$, $P_{i2}$ and $P_{i3}$ are internally-disjoint (Shown in Figure 8). It can be seen that $P_{11} \cup P_{12} \cup P_{13} \cup P_{21} \cup P_{22} \cup P_{23} \cup P_{31} \cup P_{32} \cup P_{33}$ forms a subdivision of $K_{3,3}$. Thus, $G$ is non-planar.
\end{proof}

\input{Figure8.Tpx}

We next prove Theorem~\ref{thm1.5}.

\noindent\emph{Proof of Theorem~\ref{thm1.5}}

We first characterize all 3-connected planar $Oct_{1}^{+}$-free graphs. Suppose a graph $G$ is 3-connected planar $Oct_{1}^{+}$-free graph. Two cases now arise, depending on whether G has an Oct-minor.

{\bf Case 1.} $G$ contains an $Oct$-minor.

It is clearly that $G=Oct$ is 3-connected planar $Oct_{1}^{+}$-free graph. Thus we assume that $G \neq Oct$. By Lemma~\ref{lem2.7}, $G$ has a minor $J$, which is obtained from $Oct$ by either adding an edge or 3-splitting a vertex. Adding any edge to $Oct$ results in a non-planar graph. And there are two graphs obtained by 3-splitting a vertex of $Oct$, one is $Oct_{1}^{+}$, another is non-planar. Hence, in this case, $Oct$ is the only 3-connected planar $Oct_{1}^{+}$-free graph.

{\bf Case 2.} $G$ is $Oct$-free.

Since $G$ is 3-connected, $G$ is constructed by taking 3-sums starting from graphs in $\{K_{4}\} \cup \{C_{2n-1}^{2} : n\geq 3\} \cup \{L_{4}^{'}, L_{5}, L_{5}^{'}, L_{5}^{''}, P_{10}\}$ by Theorem~\ref{thm1.2}.
By Lemma~\ref{lem3.3}, Lemma~\ref{lem3.5} and the planarity of $G$, $G$ is $L_{5}$ or $G$ is the special 3-sums of $K_{4}$. Thus $G$ belongs to $\{L_{5}\} \cup \mathscr{K}$ in this case.

It follows from Case 1 and Case 2 that $G$ belongs to $\{Oct,L_{5}\} \cup \mathscr{K}$. Then Theorem~\ref{thm1.5} follows from Lemma~\ref{lem2.6}, Lemma~\ref{lem3.3} and Lemma~\ref{lem3.4}.
$\hfill\qedsymbol$

\subsection{4-connected $Oct_{2}^{+}$-free graphs}

In this section, we characterize 4-connected $Oct_{2}^{+}$-free graphs and prove Theorem~\ref{thm1.6}.

\begin{lem}\label{lem3.6}
Graphs in $\mathcal{C}$ are all 4-connected $Oct_{2}^{+}$-free graphs.
\end{lem}

\begin{proof}
It clearly holds since $C_{2k}^{2}$ $(k \geq 3)$ is planar and  $C_{2k+1}^{2}$ $(k \geq 2)$ is $Oct$-free.
\end{proof}

\begin{lem}\label{lem3.7}
If a 4-connected graph $G \in \mathcal{L}$ and $G$ is $Oct_{2}^{+}$-free, then $G$ is planar or $G = L(K_{3,3})$.
\end{lem}

\begin{proof}
Suppose $G = L(H)$, where $H$ is a cubic cyclically 4-connected graph. If $G$ is planar, then $G$ is $Oct_{2}^{+}$-free since $Oct_{2}^{+}$ is non-planar. When $G$ is non-planar, $H$ is non-planar too. By Lemma~\ref{lem2.4}, $H = K_{3,3}$ or $H$ contains a subdivision of $V_{8}$.

{\bf Case 1.} $H = K_{3,3}$.

As shown in Figure 9, $G = L(K_{3,3})$ has 9 vertices $\{v_{1}, v_{2},..., v_{9}\}$ and 18 edges. If $Oct_{2}^{+}$ is a minor of $G$, the minor can be obtained from $G$ by contracting two edges and then deleting some edges. Without loss of generality, we first contract $v_{4}v_{5}$ to $v_{5}$ and denote the resulting graph by $G_{816}$, means it has 8 vertices and 16 edges. By symmetry, we next contract one of the edges in $\{v_{1}v_{5}, v_{2}v_{5}, v_{3}v_{6}, v_{1}v_{2}, v_{2}v_{3}, v_{5}v_{6}, v_{1}v_{3}, v_{1}v_{7}, v_{3}v_{9}, v_{2}v_{8}\}$. We verify every case in order and up to isomorphism there are six resulting graphs, we denote them by $G_{713}^{a}$, $G_{714}^{a}$, $G_{713}^{b}$, $G_{714}^{b}$, $G_{715}$, $G_{714}^{c}$ respectively. Since they are all $Oct_{2}^{+}$-free graphs, $G = L(K_{3,3})$ is $Oct_{2}^{+}$-free too.

\input{Figure9.Tpx}
\input{Figure10.Tpx}

{\bf Case 2.} $H$ contains a subdivision of $V_{8}$.

By Lemma~\ref{lem2.5}, $G$ contains $L(V_{8})$ as a minor. Since $L(V_{8})$ contains a $Oct_{2}^{+}$-minor as shown in Figure 10, $G$ contains $Oct_{2}^{+}$ as a minor.
\end{proof}

\noindent\emph{Proof of Theorem~\ref{thm1.6}}

Suppose $G$ is a 4-connected graph that is not in $\mathcal{C} \cup \mathcal{L}$. If $G$ is planar, then $G$ is $Oct_{2}^{+}$-free clearly. If $G$ is non-planar, by Theorem~\ref{thm2.2}, there is a $(G, K_{5})$-chain. That is $G$ can be generated by repeatedly splitting vertices of $C_{5}^{2}$.

Then Theorem~\ref{thm1.6} follows from Lemma~\ref{lem3.6} and Lemma~\ref{lem3.7}.
$\hfill\qedsymbol$

\end{document}